\documentclass{article} 

\usepackage{amsfonts,amsmath,latexsym}
\usepackage{amstext,amssymb,esint}
\usepackage{amsthm}
\usepackage {epsf}
\usepackage {epsfig}
\usepackage{tikz}
\usetikzlibrary{arrows, patterns}
\usepackage[scanall]{psfrag}
\usepackage {graphicx}
\usepackage[section]{placeins}
\usepackage{verbatim}
\usepackage{appendix}

\newcommand {\R}{\mathbf{R}}
\newcommand {\Ss}{\mathbf{S}}

\newcommand{\Ric}{\operatorname{Ric}}

\newtheorem {thm} {Theorem}

\newtheorem {lemma} [thm] {Lemma}

\newtheorem {rmk} {Remark}

\begin{document} 

\title{On a fourth order conformal invariant} 
\author{Jesse Ratzkin \footnote{Institut f\"ur Mathematik, Universit\"at
W\"urzburg, {\tt jesse.ratzkin@mathematik.uni-wuerzburg.de}}}

\maketitle

\begin{abstract} 
In this note we prove that a fourth order conformal invariant on 
$\Ss^1 \times \Ss^{n-1}$ can be arbitrarily close to that of the round 
$n$-dimensional sphere, generalizing a result of Schoen about the 
classical Yamabe invariant. 
\end{abstract} 

\section{Introduction} 

In this note we consider a fourth order invariant related to the classical 
Yamabe invariant. 

\subsection{Basic definitions} 

Let $n \geq 5$ and let $(M,g)$ be a compact Riemannian 
manifold without boundary, and define 
\begin{equation} \label{q_curv_defn} 
Q_g = -\frac{1}{2(n-1)} \Delta_g R_g - \frac{2}{(n-2)^2} |\Ric_g |^2 + 
\frac{n^3 - 4n^2 + 16 n - 16}{8 (n-1)^2(n-2)^2} R_g^2 ,
\end{equation}  
where $\Delta_g$ is the Laplace-Beltrami operator, $R_g$ is the 
scalar curvature, and $\Ric_g$ is the Ricci curvature. We simplify 
the expression of of $Q_g$ in \eqref{q_curv_defn} by introducing 
the Schouten tensor 
\begin{equation} \label{schouten_defn} 
A_g = \frac{1}{n-2} \left ( \Ric_g - \frac{R_g g}{2(n-1)} \right ) , 
\qquad J_g = \operatorname{tr}_g(A_g) = \frac{R_g}{2(n-1)}, 
\end{equation} 
so that 
\begin{equation} \label{q_curv_defn2} 
Q_g = - \Delta_g J_g - 2 |A_g|^2 + \frac{n}{2} J_g^2. 
\end{equation} 

The $Q$-curvature transforms nicely under a conformal change, 
namely
\begin{equation} \label{trans_rule1} 
\widetilde g = u^{\frac{4}{n-4}} g \Rightarrow 
Q_{\widetilde g} = \frac{2}{n-4} u^{- \frac{n+4}{n-4}} P_g(u), 
\end{equation} 
where $P_g$ is the Paneitz operator 
\begin{equation} \label{paneitz_op_defn} 
P_g(u) = (-\Delta_g)^2 (u) + \operatorname{div} ( 4A_g (\nabla u, 
\cdot ) - (n-2) J_g \nabla u) + \frac{n-4}{2} Q_g u. 
\end{equation} 
In general the Paneitz operator transforms according to the 
rule 
\begin{equation} \label{trans_rule2}
\widetilde g = u^{\frac{4}{n-4}} g \Rightarrow P_{\widetilde g}(v) 
= u^{-\frac{n+4}{n-4}} P_g(uv). 
\end{equation} 
One recovers \eqref{trans_rule1} by substituting $v=1$ 
into \eqref{trans_rule2}. 

S. Paneitz \cite{Pan1} introduced the operator \eqref{paneitz_op_defn} 
and explored some of its transformation properties. Afterwards 
T. Branson \cite{Bran1, Bran2} extended Paneitz' definition and studied 
the associated $Q$-curvature. The interested reader can find excellent 
surveys in \cite{BG, CEOY, HY}. 

The transformation rules \eqref{trans_rule1} and \eqref{trans_rule2} 
motivate us to define the energy 
function $\mathcal{Q} : [g] \rightarrow \R$, where $[g] = \{ 
\widetilde g = u^{\frac{4}{n-4}} g : u \in \mathcal{C}^\infty (M) , 
u>0\}$ is the conformal class of $g$, by 
\begin{equation} \label{q_energy} 
\mathcal{Q} (\widetilde g) = \frac{\int_M Q_{\widetilde g} 
d\mu_{\widetilde g}}{ (\operatorname{Vol}_{\widetilde g} (M)
)^{\frac{n-4}{n}}} = \frac{2}{n-4} \frac{\int_M u P_g(u) 
d\mu_g} {\left ( \int_M u^{\frac{2n}{n-4}} d\mu_g \right )^{\frac{n-4}{n}}} .
\end{equation}
We then the conformal invariant 
\begin{eqnarray} \label{paneitz_inv1} 
\mathcal{Y}_4^+ ([g],M) & = & \inf_{\widetilde g \in [g]}
\mathcal{Q}(\widetilde g) = \inf \left \{ \frac{\int_M 
Q_{\widetilde g} d\mu_{\widetilde g}} {(\operatorname{Vol}
_{\widetilde g}(M))^{\frac{n-4}{n}}} : \widetilde g \in [g] \right \} 
\\ \nonumber 
& = & \inf \left \{ \frac{2}{n-4} \frac{\int_M u P_g(u) d\mu_g}
{\left ( \int_M u^{\frac{2n}{n-4}} d\mu_g\right )^{\frac{n-4}{n}}} : 
u \in \mathcal{C}^\infty (M) , u>0 \right \} ,
\end{eqnarray} 
which is a fourth-order analog of the famous Yamabe invariant,
and the differential invariant 
\begin{equation} \label{paneitz_inv2} 
\mathbb{Y}_4^+(M) = \sup_{[g] \in \mathfrak{c}} 
\mathcal{Y}_4^+ ([g],M)  = \sup_{[g] \in \mathfrak{c}}
\inf_{\widetilde g \in [g]} \left \{ \frac{ \int_M 
Q_{\widetilde g} d\mu_{\widetilde g} } { ( 
\operatorname{Vol}_{\widetilde g} (M))^{\frac{n-4}{n}}} 
\right \} , 
\end{equation} 
where $\mathfrak{c}$ is the space of conformal classes on the 
manifold $M$. 
The subscript $4$ in both $\mathcal{Y}_4^+$ and in $\mathbb{Y}_4^+$ 
refers to the fact that the underlying differential operator is fourth-order, 
while the $+$ refers to the fact that we require all test functions in the 
infimum for $\mathcal{Y}_4^+$ must all be positive. Naturally one may 
also define 
$$\mathcal{Y}_4([g],M) = \inf \left \{ \frac{2}{n-4} \frac{\int_M u 
P_g(u) d\mu_g} { \left ( \int_M |u|^{\frac{2n}{n-4}} d\mu_g \right 
)^{\frac{n-4}{n}} } : u \in \mathcal{C}^\infty(M), u \not \equiv 0 
\right  \},$$ 
and clearly $\mathcal{Y}_4([g],M) \leq \mathcal{Y}_4^+([g],M)$. 

\subsection{Scalar curvature and the Yamabe invariant} 

Much of the work devoted to the Paneitz operator $P_g$ 
and its associated $Q$-curvature is motivated by results about the total 
scalar curvature functional and its associated Yamabe invariant. 

Given a compact Riemannian manifold $(M,g)$ without boundary one 
defines the total scalar curvature functional on the conformal class 
$[g]$ as
\begin{equation} \label{tot_scal_curv1} 
\mathcal{R} (\widetilde g) = \frac{\int_M R_{\widetilde g} 
d\mu_{\widetilde g}} {(\operatorname{Vol}_{\widetilde g} (M) 
)^{\frac{n-2}{n}} }. 
\end{equation} 
One can simplify this expression using the transformation rule 
\begin{equation} \label{trans_rule3} 
\widetilde g = u^{\frac{4}{n-2}} g \Rightarrow R_{\widetilde g} 
= \frac{n-2}{4(n-1)} u^{-\frac{n+2}{n-2}} \mathcal{L}_g(u),
\end{equation} 
where $\mathcal{L}_g$ is the conformal Laplacian 
\begin{equation} \label{conf_lap} 
\mathcal{L}_g = -\Delta_g  + \frac{4(n-1)}{n-2} R_g ,
\end{equation} 
which enjoys the transformation rule 
\begin{equation} \label{trans_rule4} 
\widetilde g = u^{\frac{4}{n-2}} g \Rightarrow
\mathcal{L}_{\widetilde g} (v) = u^{-\frac{n+2}{n-2}} 
\mathcal{L}_g (uv). 
\end{equation} 
Observe that these transformation rules mean we can 
rewrite $\mathcal{R}$ as 
\begin{equation} \label{tot_scal_curv2} 
\mathcal{R}(u^{\frac{4}{n-2}} g ) = \frac{n-2}{4(n-1)} 
\frac{\int_M u \mathcal{L}_g(u) d\mu_g}{\left ( \int_M
u^{\frac{2n}{n-2}} d\mu_g \right )^{\frac{n-2}{n}}}.
\end{equation} 

The classical Yamabe invariants are 
\begin{eqnarray} \label{yam_inv1} 
\mathcal{Y}([g],M) & = & \inf_{\widetilde g \in [g]} \mathcal{R}
(\widetilde g) = \inf \left \{ \frac{\int_M R_{\widetilde g} d\mu_{\widetilde g}}
{(\operatorname{Vol}_{\widetilde g} (M))^{\frac{n-2}{n}}} : \widetilde g 
\in [g] \right \} \\ \nonumber 
& = & \inf \left \{ \frac{n-2}{4(n-1)} \frac{\int_M u \mathcal{L}_g(u) 
d\mu_g}{\left ( \int_M u^{\frac{2n}{n-2}} d\mu_g \right )^{\frac{n-2}{n}}}
: u \in \mathcal{C}^\infty(M) , u>0 \right \}
\end{eqnarray}
and 
\begin{equation} \label{yam_inv2} 
\mathbb{Y}(M) = \sup_{[g] \in \mathfrak{c}} \mathcal{Y}([g], M) 
= \sup_{[g] \in \mathfrak{c}} \inf_{\widetilde g \in [g]} 
\frac{\int_M R_{\widetilde g} d\mu_{\widetilde g}}{ (
\operatorname{Vol}_{\widetilde g} (M) )^{\frac{n-2}{n}}} . 
\end{equation} 
In contrast to the fourth-order case, in this situation the 
maximum principle implies 
$$\mathcal{Y}([g],M) = \inf  \left \{ \frac{n-2}{4(n-1)} \frac{\int_M u \mathcal{L}_g(u) 
d\mu_g}{\left ( \int_M |u|^{\frac{2n}{n-2}} d\mu_g \right )^{\frac{n-2}{n}}}
: u \in \mathcal{C}^\infty(M) , u \not \equiv 0 \right \} . 
$$
In particular, minimizing the functional $\mathcal{R}$ over 
all nontrivial functions in $W^{1,2}(M)$ will automatically yield 
a positive minimizer. On the other hand, minimizers of $\mathcal{Y}_4
([g],M)$, if they exist, might change sign. 

Yamabe \cite{Y} first defined these two invariants while investigating 
the the problem of finding a constant scalar curvature metric 
in a given conformal class. Aubin 
\cite{Aub} proved that $\mathcal{Y}([g],M) \leq \mathcal{Y}
([g_0],\Ss^n)$, where $g_0$ is the round metric on the sphere $\Ss^n$, 
and proved that if $\mathcal{Y}([g],M) < \mathcal{Y}([g_0],M)$ there 
exists a smooth, constant scalar curvature metric $\widetilde g \in [g]$
such that $\mathcal{R}(\widetilde g) = \mathcal{Y}([g],M)$. 
In \cite{Sch} Schoen completed Yamabe's program, proving 
that $\mathcal{Y}([g], M) < \mathcal{Y}([g_0], \Ss^n)$ for each 
conformal class $[g] \neq [g_0]$. In particular, $\mathcal{Y}([g],M) 
< \mathcal{Y}([g_0], \Ss^n)$ whenever $M$ is not the sphere. 

On the other hand, the equality $\mathbb{Y}(M) = \mathbb{Y}(\Ss^n)$ 
may occur even when $M$ is not the sphere. In particular, 
Schoen \cite{Sch_var} found an explicit sequence of metrics 
$g_k$ on the product $\Ss^1 \times \Ss^{n-1}$ such that 
$\mathcal{Y}([g_k], \Ss^1 \times \Ss^{n-1}) \rightarrow \mathcal{Y}
([g_0], \Ss^n)$ as 
$k \rightarrow \infty$, and so $\mathbb{Y}\mathbb(\Ss^1 
\times \Ss^{n-1}) = \mathbb{Y}(\Ss^n)$. As the underlying manifolds 
are not diffeomorphic, the equality above cannot be realized by 
a smooth metric on $\Ss^1 \times \Ss^{n-1}$. 

\subsection{Previous results and our main theorem} 

We summarize some previous theorems regarding the invariant 
$\mathcal{Y}_4^+ ([g],M)$. Esposito 
and Robert \cite{ER} showed that $\mathcal{Y}_4^+([g],M)$ is 
finite for each conformal class $[g]$ on $M$. 
To state the next result we define 
$$\mathcal{Y}_4^* ([g],M) = \inf_{\widetilde g \in [g], R_{\widetilde g} > 0}
\mathcal{Q} (\widetilde g).$$
Gursky, Hang and Lin \cite{GHL} proved that if $n= \dim(M) \geq 6$ and 
if $\mathcal{Y}([g],M) > 0$ 
and $\mathcal{Y}_4^*([g],M)>0$ then 
$$\mathcal{Y}_4([g],M) = \mathcal{Y}_4^+ ([g],M) = \mathcal{Y}_4^* 
([g],M) .$$
Shortly thereafter Hang and Yang \cite{HY} proved that if $\mathcal{Y}([g],M)>0$ 
and $Q_g \geq 0$ with $Q_g \not \equiv 0$ then 
$$\mathcal{Y}_4([g],M) = \mathcal{Y}_4^+([g],M) \leq \mathcal{Y}_4 
([g_0], \Ss^n),$$ 
and that equality in the last inequality implies $[g] = [g_0]$. Moreover, 
under these hypotheses there exists a smooth, constant $Q$-curvature 
metric $\widetilde g \in [g]$ such that $\mathcal{Q}(\widetilde g) = 
\mathcal{Y}_4^+ ([g],M)$. 

Our main result is the following theorem. 
\begin{thm} \label{main_thm} 
There exists a sequence of metrics $g_k$ on the product 
$\Ss^1 \times \Ss^{n-1}$ such that $\mathcal{Y}_4^+ ([g_k], 
\Ss^1 \times \Ss^{n-1}) 
\rightarrow \mathcal{Y}_4^+ ([g_0], \Ss^n)$, where $g_0$ is the standard 
round metric on $\Ss^n$. As a consequence $\mathbb{Y}_4^+ 
(\Ss^1 \times \Ss^{n-1}) = \mathbb{Y}_4^+ (\Ss^n)$. 
\end{thm}

\begin{rmk} The theorem of Hang and Yang \cite{HY} 
referenced above implies 
the equality $\mathbb{Y}_4^+ (\Ss^{n-1} \times \Ss^1) = \mathbb{Y}_4^+
(\Ss^n)$ cannot be realized by a smooth metric on $\Ss^{n-1} \times 
\Ss^1$.  
\end{rmk}

We base our proof of Theorem \ref{main_thm} on the explicit 
examples of the Delaunay metrics recently discovered by 
Frank and K\"onig \cite{FK}, following the example of 
R. Schoen \cite{Sch_var}. 

\section{Proof of our main theorem} 

In this section we present a proof of Theorem \ref{main_thm}, 
using the Delaunay metrics of Frank and K\"onig as our sequence 
of metrics. We first present some preliminary facts we require 
in our proof, and then carefully describe the Delaunay metrics, 
verifying some of their properties. Finally we complete the 
proof of Theorem \ref{main_thm}. 

\subsection{Preliminaries} 

We begin with the well-known variational characterization of 
constant $Q$-curvature metrics. One can find the following 
computation in \cite{Rob}, among other places, but we 
include it for the reader's convenience. 

It will be convenient to let $p^\# = \frac{2n}{n-4}$, let $\| \cdot \|_p$
denote the $L^p$-norm on $(M, d\mu_g)$, and define the bilinear 
form $\mathcal{E} (u,v) = \int_M v P_g(u) d\mu_g$. Observe that 
\begin{eqnarray} \label{paneitz_bilin_form} 
\mathcal{E} (u,v) & = & \int_M v P_g(u) d\mu_g \\ \nonumber 
& = & \int_M v \left ( \Delta_g ^2 u+ 4 \operatorname{div} (A_g(\nabla u, \cdot)
- (n-2) \operatorname{div} (J_g \nabla u ) + \frac{n-4}{2} Q_g u \right ) d\mu_g
\\ \nonumber 
& = & \int_M \Delta_g v \Delta_g u - 4 A_g(\nabla u, \nabla v) + (n-2) J_g 
\langle \nabla u, \nabla v \rangle  + \frac{n-4}{2} Q_g uv d\mu_g 
\\ \nonumber 
& = & \int_M u \left ( \Delta_g ^2 v+ 4 \operatorname{div} (A_g(\nabla v, \cdot)
- (n-2) \operatorname{div} (J_g \nabla v ) + \frac{n-4}{2} Q_g v \right ) d\mu_g
\\ \nonumber 
& = & \mathcal{E} (v,u) ,
\end{eqnarray} 
and so $\mathcal{E}$ is symmetric. We denote $\mathcal{E}(u,u) 
= \mathcal{E}(u)$. 

\begin{lemma} 
Let $W^{2,2}_+(M)$ denote the subspace of (almost everywhere) 
positive functions in $W^{2,2}(M)$. The functional 
$$ W^{2,2}_+(M) \ni 
u \mapsto \mathcal{Q}(u^{\frac{4}{n-4}} g)  $$
is differentiable and its total derivative is 
\begin{equation} \label{derivative_paneitz_func}
D\mathcal{Q}(u) (v) = \frac{4}{(n-4) \| u \|_{p^\#}^2} \int_M
v \left ( P_g(u) - \| u \|_{p^\#}^{-p^\#} \mathcal{E}(u) u^{\frac{n+4}{n-4}}
\right ) d\mu_g .
\end{equation} 
\end{lemma}

\begin{proof} 
Let $u \in W^{2,2}_+(M) \cap \mathcal{C}^0(M)$ and choose $v \in W^{2,2}(M)$ such that 
$0<\sup |v| < \frac{1}{2} \inf u$, which is possible because $M$ is a 
compact manifold. In particular, $u+tv \in W^{2,2}_+(M)$ for 
$0<t<3/2$. Expanding \eqref{paneitz_bilin_form} we obtain 
\begin{eqnarray} \label{derivative1} 
\mathcal{E}(u+tv) & = & \mathcal{E}(u) + 2t\mathcal{E}(u,v) + t^2 
\mathcal{E}(v) \\ \nonumber 
& = & \int_M u P_g (u) d\mu_g + 2t \int_M v P_g(u) d\mu_g 
+ t^2 \int_M v P_g(v) d\mu_g. \end{eqnarray} 
On the other hand we have 
\begin{eqnarray} \label{derivative2} 
\left ( \int_M (u+tv)^{\frac{2n}{n-4}} d\mu_g \right )^{\frac{4-n}{n}} & = & 
\left ( \int_M u^{\frac{2n}{n-4}} d\mu_g \right )^{\frac{4-n}{n}}  \\ \nonumber 
&& -2t \left ( \int_M u^{\frac{2n}{n-4}} d\mu_g \right )^{\frac{4-2n}{n}} 
\left ( \int_M v u^{\frac{n+4}{n-4}} d \mu_g \right ) + 
\mathcal{O} (t^2 \| v \|_{W^{2,2}(M)}^2) \\ \nonumber 
& = & \| u \|_{p^\#}^{-2} \left ( 1 - 2t \| u \|_{p^\#}^{-p^\#} \int_M 
v u^{\frac{n+4}{n-4}} d\mu_g + \mathcal{O} (t^2 \| v \|_{W^{2,2}(M)}) 
\right ) . 
\end{eqnarray} 
Combining \eqref{derivative1} and \eqref{derivative2} gives 
$$ \mathcal{Q}(u+tv) = \frac{2}{(n-4) \| u \|_{p^\#}^2} \left ( \mathcal{E}(u) 
+ 2t \int_M v \left ( P_g(u) - \mathcal{E}(u) \| u \|_{p^\#}^{-p^\#} u^{\frac{n+4}{n-4}} 
\right ) d\mu_g  + \mathcal{O} (t^2 \| v\|_{W^{2,2}(M)}^2 ) \right ) , $$ 
which completes the proof.
\end{proof} 

\begin{lemma} The metric $\widetilde g = u^{\frac{4}{n-4}} g$ is a critical 
point of the functional $\mathcal{Q}$ if and only if its $Q$-curvature 
$Q_{\widetilde g}$ is constant. 
\end{lemma} 

\begin{proof} First suppose $\widetilde g$ is a critical point of $\mathcal{Q}$, 
which implies 
$$\left . \frac{d}{dt} \right |_{t=0} \mathcal{Q}((u+tv)^{\frac{4}{n-4}} g) =0$$ 
for each sufficiently small $v$, and so \eqref{derivative_paneitz_func} gives 
$$ P_g(u) = \frac{\mathcal{E}(u) u^{\frac{n+4}{n-4}}}{\| u \|_{p^\#}^{p^\#}}.$$
Substituting this value for $P_g (u)$ into \eqref{trans_rule1} gives 
$$Q_{\widetilde g} = \frac{2}{(n-4) \| u \|_{p^\#}^{p^\#}} \mathcal{E}(u),$$
which is indeed a constant. 

On the other hand suppose $Q_{\widetilde g} = \overline {Q}$ is a 
constant, in which case \eqref{trans_rule1} gives 
$$P_g(u) = \frac{n-4}{2} \overline{Q} u^{\frac{n+4}{n-4}} \Rightarrow 
\mathcal{E}(u)  = \frac{n-4}{2} \overline{Q} \| u \|_{p^\#}^{p^\#}.$$
Substituting this value of $\mathcal{E}(u)$ 
into \eqref{derivative_paneitz_func} gives 
\begin{eqnarray*} 
\left. \frac{d}{dt} \right |_{t=0} \mathcal{Q} ((u+tv)^{\frac{4}{n-4}} g) & = & 
\frac{4}{(n-4) \| u\|_{p^\#}^2} \int_M v \left ( P_g(u) - \frac{n-4}{2} 
\overline{Q} u^{\frac{n+4}{n-4}} \right ) d\mu_g \\ 
& = & \frac{4}{(n-4) \| u \|_{p^\#}^2}\int_M v \left ( \frac{n-4}{2} 
\overline{Q} u^{\frac{n+4}{n-4}} - \frac{n-4}{2} \overline{Q} 
u^{\frac{n+4}{n-4}} \right ) d\mu_g \\ 
& = & 0 , 
\end{eqnarray*} 
and so $\widetilde g = u^{\frac{4}{n-4}} g$ is indeed a critical point of 
$\mathcal{Q}$. 
\end{proof} 

\subsection{Delaunay metrics} 
We seek constant $Q$-curvature metrics of the form $g_v 
= v^{\frac{4}{n-4}} (dt^2 + d\theta^2)$ on the cylinder $\R 
\times \Ss^{n-1}$, where $d\theta^2$ is the standard round metric 
on $\Ss^{n-1}$. It will be convenient to normalize the $Q$-curvature to 
be $\frac{n(n^2-4)}{8}$, the same value as the $Q$-curvature of the 
round metric on $\Ss^n$. By \eqref{trans_rule1}, the condition 
$Q_{g_v} = \frac{n(n^2-4)}{8}$ is equivalent to 
the partial differential equation 
\begin{eqnarray} \label{paneitz_pde1} 
\frac{n(n-4)(n^2-4)}{16} v^{\frac{n+4}{n-4}} & = & \frac{\partial^4 v}
{\partial t^4} - \left ( \frac{n(n-4) + 8}{2} \right ) \frac{\partial^2 v}
{\partial t^2} + \frac{n^2(n-4)^2}{16} v \\ \nonumber 
& = & \Delta_\theta^2 + 2 \Delta_\theta \frac{\partial^2 v}{\partial t^2} 
- \frac{n(n-4)}{2} \Delta_\theta v 
\end{eqnarray} 
where $\Delta_\theta$ is the Laplace-Beltrami metric on the 
sphere $\Ss^{n-1}$ with respect to the standard round metric. 

We restrict our attention to rotationally invariant metrics, 
so that $v$ is  a function of $t$ alone and \eqref{paneitz_pde1} 
becomes 
\begin{equation} \label{paneitz_ode1} 
\ddddot v - \left ( \frac{n(n-4)+8}{2} \right ) \ddot v + 
\frac{n^2(n-4)^2}{16} v = \frac{n(n-4)(n^2-4)}{16} v^{\frac{n+4}{n-4}}. 
\end{equation} 
We immediately find two special solutions, namely the cylindrical 
and spherical solutions 
\begin{equation} 
v_{cyl} = \left ( \frac{n(n-4)}{n^2-4} \right )^{\frac{n-4}{8}}, \qquad 
v_{sph} = (\cosh t )^{\frac{4-n}{2}}. 
\end{equation} 
Observe that, since $n>4$ we have $0 < v_{cyl} < 1$ and 
$$v_{sph} (0) = 1 = \max (v_{sph} (t)), \qquad \dot v_{sph}(t) < 0 
\textrm{ for }t>0, \qquad \dot v_{sph} (t) >0 \textrm{ for }t< 0.$$

Frank and K\"onig recently classified all positive global solutions 
of the ODE \eqref{paneitz_ode1}, proving there exists a periodic 
solution $v_a$ for each $a \in (v_{cyl},1)$ attaining its maximal 
value of $a$ when $t=0$. Moreover, they show any global, 
positive solution of \eqref{paneitz_pde1} must either have the 
form $v(t) = v_a(t+T)$ or $v(t) = (\cosh (t+T))^{\frac{4-n}{2}}$ for 
some $T \in \R$, 
or $v \equiv \left ( \frac{n(n-4)}{n^2-4} \right )^{\frac{n-4}{8}}$. We 
call $g_{v_a} = v_a^{\frac{4}{n-4}} (dt^2 + d\theta^2)$ the 
{\bf Delaunay metric} with Delaunay parameter $a$. 

Each solution $v_a$ is periodic with period $T_a$, attains its maximal 
value at each integer multiple of $T_a$, attains its minimal value at 
each half-integer multiple of $T_a$, and is symmtric about each 
of its critical points. Moreover, the period $T_a$ is an increasing 
function of $a$ with $\lim_{a \nearrow 1} T_a = \infty$ and 
$\lim_{a \searrow v_{cyl}} T_a = T_{cyl}$, where $T_{cyl}$ is 
the formal period of $v_{cyl}$, given by 
\begin{equation} \label{cyl_period} 
T_{cyl} = \frac{2\pi}{\mu}, \qquad \mu = \frac{1}{2} \sqrt{ 
\sqrt{n^4 - 64 n + 64} - n(n-4) + 8}. 
\end{equation} 
The period $T_{cyl}$ is the fundamental period of the linearization 
of the operator $P_g$, linearized about the cylindrical solution (see 
Section 3.3 of \cite{R}). One can also show $\sup v_a(t) 
\leq 1$ for each Delaunay parameter $a$. We 
let $\epsilon (a) = \min_{t \in \R} v_a(t)$. 

We define the energy 
\begin{equation} \label{del_energy} 
\mathcal{H} (v) = - \dot v \dddot v + \frac{1}{2} (\ddot v)^2 + \left ( 
\frac{n(n-4)+8}{4} \right ) \dot v^2 - \frac{n^2(n-4)^2}{32} v^2 + 
\frac{(n-4)^2(n^2-4)}{32} v^{\frac{2n}{n-4}}.
\end{equation} 
Differentiating $\mathcal{H}$ with respect to $t$ we find 
$$ \frac{d}{dt} \mathcal{H} = 
-\dot v \left ( \ddddot v - \left ( \frac{n(n-4)+8}{2} \right ) \ddot v + 
\frac{n^2(n-4)^2}{16} - \frac{n(n-4)(n^2-4)}{16} v^{\frac{n+4}{n-4}} 
\right ) ,$$
and so $\mathcal{H}(v)$ is constant if $v$ satisfies \eqref{paneitz_ode1}. 
Evaluating this energy on the cylindrical and spherical 
solutions we find 
\begin{equation} \label{cyl_sph_energy} 
\mathcal{H}_{cyl} = \mathcal{H} (v_{cyl}) = - \frac{n(n-4)^2}{8} 
\left ( \frac{n(n-4)}{n^2-4} \right )^{\frac{n-4}{4}} < 0, \qquad 
\mathcal{H}_{sph} = \mathcal{H} (v_{sph}) = 0. 
\end{equation} 

Restricting attention to the $(v,\dot v)$ in phase space we see 
that the level set $\{ \mathcal{H} = 0 \} \cap \{ \ddot v = 0, 
\dddot v = 0\}$ consists entirely of the solution curve of $v_{sph}$ 
together with the point $(0,0)$. For each 
$$0 < H < -\mathcal{H}_{cyl} = \frac{n(n-4)^2}{8} 
\left ( \frac{n(n-4)}{n^2-4} \right )^{\frac{n-4}{8}}$$ 
the level set $\{ \mathcal{H} = - H\} \cap \{ \ddot v = 0, \dddot v = 0\}$ is 
a closed curve associated to the Delaunay solution $v_a$ for 
some $a \in (v_{cyl}, 1)$. Combining Theorems 1, 2, and 3 
of \cite{vdB} we find that these solution curves do not cross and 
and that the energy level completely determines the Delaunay solution. 
In particular, we see 
that $\lim_{a \nearrow 1} \epsilon(a) = 0$. We sketch some of these 
solution curves in the phase plane in Figure \ref{phase_plane_fig} 
below. 

\begin{figure} [h] 
\centering
\begin{tikzpicture}
\coordinate (0) at (-2,-2); 
\draw[->] (0,1) -- (8,1) coordinate[label = {below:$v$}] (xmax);
\draw[->] (1,-2) -- (1,4) coordinate[label = {right:$\dot v$}] (ymax);
\draw[thick, blue,->>] (1,1) .. controls (2,4) and (5.5,4) .. (6,1);
\draw[thick, blue, ->>] (6,1) .. controls (5.5,-2) and (2,-2).. (1,1);
\draw[thick,red,->>] (2,1) .. controls (2.5,3) and (4.5,3) .. (5,1);
\draw[thick,red,->>] (5,1) .. controls (4.5,-1) and (2.5,-1) .. (2,1); 
\draw[thick, green,->>] (2.8,1) .. controls (3.1,2) and (4,2) .. (4.2,1); 
\draw[thick,green,->>] (4.2,1) .. controls (4,0) and (3.1,0) .. (2.8,1); 
\node at (3.5,1) {$*$};  
\node at (6.8,-0.8) {cylindrical solution}; 
\draw [->] (6.5,-0.5) -- (3.55,1); 
\draw [->] (7,4) -- (5.5,2.3); 
\node at (7.1,4.2) {spherical solution}; 
\end{tikzpicture} 
\caption{This figure shows the level curves of $\mathcal{H}$ 
in the $(v,\dot v)$ phase-plane.} \label{phase_plane_fig}
\end{figure}
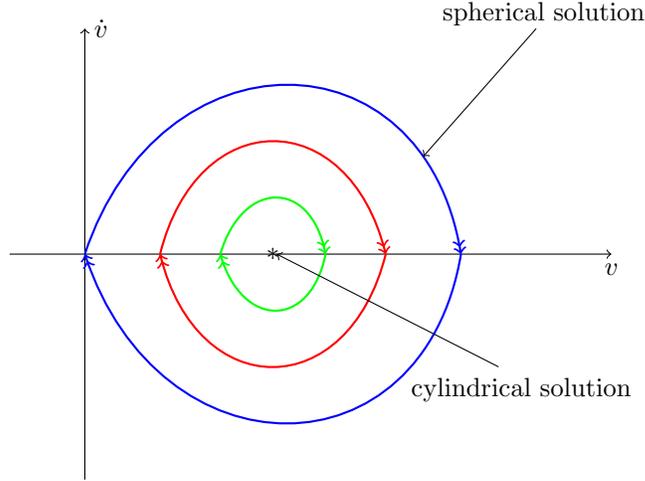

\subsection{Completion of the proof}

For each $T>0$ we consider metrics of the form $g_v = 
v^{\frac{4}{n-4}} (dt^2 + d\theta^2)$ on $\Ss^1_T \times 
\Ss^{n-1}$, identifying $\Ss^1_T$ with the 
interval $[-T/2,T/2]$. Observe that the metric $g_1 = dt^2 + 
d\theta^2$ has constant positive scalar curvature equal to 
$(n-1)(n-2)$, as well as positive $Q$-curvature $\frac{(n-1)((n-1)^2-4)}{8}$, 
and so we may apply the theorem of Hang and Yang to 
conclude 
$$\mathcal{Y}_4^+ ([dt^2 + d\theta^2], \Ss^1_T \times \Ss^{n-1})
= \mathcal{Y}_4 ([dt^2+d\theta^2], \Ss^1_T \times \Ss^{n-1}) < 
\mathcal{Y}_4^+ ( [g_0], \Ss^n).$$
Each critical point of $\mathcal{Q}$ 
in the conformal class $[dt^2 + d\theta^2]$ must be a 
constant $Q$-curvature metric on $\Ss^1 \times \Ss^{n-1}$. 
We pull this constant $Q$-curvature metric on $\Ss^1_T \times \Ss^{n-1}$ 
back to the universal cover $\R \times \Ss^{n-1}$, obtaining a 
smooth, positive. $T$-periodic function 
$v : \R \times \Ss^n \rightarrow (0, \infty)$
satisfying \eqref{paneitz_pde1}. As we discussed above, Frank and 
K\"onig classified these solutions as either the constant 
cylindrical solution $v_{cyl}$, translates of the spherical 
solution $v_{sph}$, or translates of a Delaunay solution 
$v_a$ for some $a \in (v_{cyl}, 1)$. 

The number of constant $Q$-curvature metrics in the 
conformal class $[dt^2 + d\theta^2]$ on $\Ss^1_T \times 
\Ss^{n-1}$ depends on $T$ in the following way. As in 
our previous discussion, we normalize the value of the 
$Q$-curvature to be $\frac{n(n^2-4)}{8}$. The cylindrical 
solution $v_{cyl}$ is the only solution when $0 < 
T \leq T_{cyl}$, where $T_{cyl}$ is given in \eqref{cyl_period}. 
For $T_{cyl} < T \leq 2T_{cyl}$ we have two constant $Q$-curvature 
metrics, namely the cylinder and the Delaunay metric with Delaunay 
parameter $a$ such that $T = T_a$. When $2T_{cyl} < T\leq 3T_{cyl}$ 
we obtain $3$ constant $Q$-curvature metrics, namely the 
cylindrical solution $v_{cyl}$, the Delaunay solution $v_a$ 
such that $T_a = T$, and the Delaunay solution $v_\alpha$ 
such that $T_\alpha = T/2$. Continuing inductively, when 
$(k-1)T_{cyl} < T \leq k T_{cyl}$ we obtain $k$ distinct  
constant $Q$-curvature metrics, namely the cylindrical 
solution $v_{cyl}$ together with the Delaunay solution $v_{a_l}$
with $T_{a_l} = T/l$ for each $l = 1,2,\dots,k-1$. 

For each $T> T_{cyl}$ the Delaunay solution $v_{a_1}$ such that 
$T_a = T$ solves the initial value problem $v_a(0) = a$, 
$\dot v_a(0) = 0$. By the results in \cite{vdB} these two initial 
conditions actually uniquely determine a solution 
of \eqref{paneitz_ode1}. Combining this uniqueness of 
with the fact that $\lim_{a \nearrow 1} T_a = \infty$ we 
conclude $v_a \rightarrow v_{sph} = (\cosh t)^{\frac{4-n}{2}}$ 
as $a \nearrow \infty$. 
Moreover, because each $\| v_a \|_\infty \leq 1$, this 
convergence is uniform on compact subsets by the Arzela-Ascoli 
theorem. 

Next we show that $v_{a_1}$ is the only stable critical 
point of $\mathcal{Q}$ among $\{ v_{cyl}, v_{a_1}, v_{a_2}, 
\dots, v_{a_{k-1}} \}$. The function $w_{a_l} = \dot v_{a_l}$ 
satisfies $L_{a_l} (w_{a_l}) = 0$, where $L_a$ is the linearization 
of \eqref{paneitz_pde1} about $v_{a_l}$. Observe that 
$$\{ t \in [-T/2, T/2] : w_{a_l} > 0\} = \bigcup_{j=-\lfloor l/2 
\rfloor}^{\lfloor l/2 \rfloor}  \left ( j T_{a_l} , \left ( \frac{2j+1}{2} \right ) 
T_{a_l} \right ),$$
where $\lfloor l/2 \rfloor$ denotes the greatest non-negative integer 
less than or equal to $l/2$. When $l \geq 2$ the number of nodal 
domains combined with Strum-Liouville theory implies $-L_{a_l}$ 
has at least $l$ negative eigenvalues, and so $v_{a_l}$ cannot be a 
stable critical point of $\mathcal{Q}$.  Furthermore the function 
$w_0 = \cos (\mu t)$, where $\mu$ is given by \eqref{cyl_period}, 
satisfies $L_{cyl} (w_0) =0$, where $L_{cyl}$ is the linearization 
of \eqref{paneitz_pde1} about $v_{cyl}$. When $T> 2 T_{cyl}$ the 
function $w_0$ has at least $2$ disjoint regions on which it is 
positive, so $v_{cyl}$ cannot be a stable critical point of $\mathcal{Q}$ 
for large values of $T$. 

We conclude that $v_a$ minimizes $\mathcal{Q}$ over 
the conformal class $[dt^2 + d\theta^2]$ on $\Ss^1_{T_a} 
\times \Ss^{n-1}$, and so 
\begin{eqnarray} \label{del_tot_q_curv}
\mathcal{Y}_4^+ ([dt^2 + d\theta^2], \Ss^1_{T_a} 
\times \Ss^{n-1}) & = & 
\mathcal{Q}(g_{v_a}) = \frac{2}{n-4} \frac{\int_{\Ss^1 
\times \Ss^{n-1}} Q_{g_{v_a}} 
d\mu_{g_{v_a}} }{(\operatorname{Vol}_{g_{v_a}}(\Ss^1 
\times \Ss^{n-1}) )^{\frac{n-4}{n}}} \\ \nonumber 
& = & \frac{2}{n-4} \cdot \frac{n(n^2-4)}{8}  \frac 
{\operatorname{Vol}_{g_{v_a}}(\Ss^1 \times \Ss^{n-1}))}
{(\operatorname{Vol}_{g_{v_a}}(\Ss^1 
\times \Ss^{n-1}) )^{\frac{n-4}{n}}} \\ \nonumber 
& = & \frac{n(n^2-4)}{4(n-4)} \frac{ \int_{-T_a/2}^{T_a/2} \int_{\Ss^{n-1}}
v^{\frac{2n}{n-4}} d\theta dt} { \left ( \int_{-T_a/2}^{T_a/2} 
\int_{\Ss^{n-1}} v^{\frac{2n}{n-4}} d\theta dt \right )^{\frac{n-4}{n}}} 
\\ \nonumber  
& = & \frac{n(n^2-4)}{4(n-4)} |\Ss^{n-1}|^{n/4}
\left ( \int_{-T_a/2}^{T_a/2} v_a^{\frac{2n}{n-4}} dt \right )^{n/4} .
\end{eqnarray}

Finally, we let $a \nearrow 1$ in \eqref{del_tot_q_curv} to see 
\begin{eqnarray*} 
\mathbb{Y}_4^+ (\Ss^1 \times \Ss^{n-1}) & \geq & 
\lim_{a \nearrow 1} \mathcal{Y}_4^+ ([dt^2 + d\theta^2], 
\Ss^1_{T_a} \times \Ss^{n-1}) \\ 
& = & \frac{n(n^2-4)}{4(n-4)} |\Ss^{n-1}|^{n/4} 
\left ( \int_{-\infty}^\infty (v_{sph} (t) )^{\frac{2n}{n-4}} dt 
\right )^{n/4} \\ 
& = & \frac{n(n^2-4)}{4(n-4)} |\Ss^{n-1}|^{n/4} \left ( 
\int_\R (\cosh t)^{-n} dt \right )^{n/4} \\ 
& = & \frac{n(n^2-4)}{4(n-4)} |\Ss^{n-1}|^{n/4} 
\left ( \int_0^\infty \left ( \frac{1+r^2}{2} \right )^{-n} r^{n-1}
dr \right )^{n/4}\\ 
& = & \mathcal{Q} (g_0) = \mathcal{Y}_4^+ ([g_0], 
\Ss^n) = \mathbb{Y}_4^+(\Ss^n),
\end{eqnarray*}  
where $r = e^{-t}$. This completes our proof.  \hfill $\square$


\end {document}